\documentclass[10pt]{amsart}

\usepackage{amssymb,amsthm,amsmath}
\usepackage[numbers,sort&compress]{natbib}
\usepackage{color}
\usepackage{graphicx}
\usepackage{tikz}
\usepackage{amssymb,amsthm,amsmath}
\usepackage[numbers,sort&compress]{natbib}
\usepackage{color}
\usepackage{graphicx}
\usepackage{tikz}
\usepackage{amsmath}
\usepackage{amscd}
\usepackage{amsthm}
\usepackage{amssymb}
\usepackage{latexsym}
\usepackage{enumerate}

\title[ ]{Growth of the eigensolutions of Laplacians  on   Riemannian manifolds I: construction of energy function}

\author{Wencai Liu}
\address[Wencai Liu]{Department of Mathematics, University of California, Irvine, California 92697-3875, USA}\email{liuwencai1226@gmail.com}

\theoremstyle{plain}
\newtheorem{theorem}{Theorem}[section]
\newtheorem{corollary}[theorem]{Corollary}
\newtheorem{lemma}[theorem]{Lemma}

\theoremstyle{definition}

\newtheorem{remark}[theorem]{Remark}

\begin{document}


\begin{abstract}
In this paper, we consider the   eigen-solutions of  $-\Delta u+ Vu=\lambda u$,
where   $\Delta$ is  the Laplacian  on  a    non-compact complete   Riemannian manifold.
We develop Kato's methods on manifold  and
establish the growth of the eigen-solutions   as $r$ goes to infinity based on  the  asymptotical   behaviors   of $\Delta r$ and $V(x)$, where   $r=r(x)$ is  the distance function on the manifold.
As   applications,  we   prove  several criteria of absence of eigenvalues of Laplacian,
including a new proof of the  absence of    eigenvalues  embedded into the essential spectra of free Laplacian if the radial curvature of the manifold satisfies  $ K_{\rm rad}(r)= -1+\frac{o(1)}{r}$.
\end{abstract}
\maketitle
\section{Introduction and main results}
 Let $(M,g)$ be a   connected  $n$-dimensional noncompact complete Riemannian manifold ($n\geq 2$).
The Laplace-Beltrami operator   on   $(M,g)$ is essentially self-adjoint on $C^{\infty}_0(M)$.  
We  denote the self-adjoint extension by $\Delta$ (the Laplacian).

Assume there exists $U\subset M$ such that  $M\backslash U$  is  connected
 and the induced outward normal exponential map $\exp_{\partial U}^{\perp} : N^{+}(\partial U) \to M - \overline{U}$   is a diffeomorphism, where $N^{+}(\partial U) = \{ v\in T(\partial U) \mid v {\rm ~is~outward~normal~to~}\partial U \}$.
As in \cite{kumura2010radial,kumuraflat}, let $r$ be the distance function from $\partial U$ defined on   $M-\overline{U}$.

 We   are interested in the spectral theory of $\Delta$ and  asymptotic behavior of  the eigen-solutions of  
 \begin{equation}\label{Geigen}
    -\Delta u+Vu=\lambda u
 \end{equation}
as $r(x)$ goes to infinity.

For Euclidean space $\mathbb{R}^n$, that is $M=\mathbb{R}^n$, there are rich results about spectral theory of $-\Delta +V$ with decaying potential $V$.
 A typical application of Weyl's  theorem states  that  the essential spectrum $\sigma_{\rm ess}(-\Delta+V)=[0,\infty)$ if $\limsup |V(x)|=0$. It is interesting to investigate if there exists eigenvalue embedded  into the essential spectrum.
 Kato \cite{kato}   addressed this problem and showed that   there is no eigenvalue $\lambda>a^2$ if $\limsup |x||V(x)|=a$.
 This implies there is no eigenvalue embedded into  the essential spectrum if
  $V(x)=\frac{o(1)}{1+|x|}$. There is an alternative proof similar to Kato's  by Vakulenko \cite{va}.
  By Neuman-Wigner type functions \cite{von1929uber},
$V(x)=\frac{o(1)}{1+|x|}$ is optimal so that $V(x)=\frac{o(1)}{1+|x|}$ is a spectral transition for eigenvalue embedded into the essential spectrum.
For more   examples  about (finite or dense) eigenvalues embedded into essential spectrum, see \cite{sim1,sim2,nab}.
Under stronger assumption on the perturbation, for example $V(x)=\frac{O(1)}{(1+|x|)^{1+\delta}}$ for some $\delta>0$ or   $V\in L^p(\mathbb{R}^n)$ for proper $p>0$,  the limiting absorption principle holds, originally from Agmon's theory \cite{agmon1975,sim3}.   Thus  operator  $-\Delta+V$  has no singular continuous spectrum. See the survey paper \cite{schlag2007dispersive}
for more details.
For one dimensional case, there are more results. For example, $V(x)=\frac{O(1)}{1+|x|}$ is a spectral transition for singular continuous spectrum embedded  into the essential spectrum \cite{kiselev2005imbedded,MR2307748}.
Agmon \cite{ag} and Simon \cite{sim4}, using Kato's methods, independently
 obtained the quantitative bounds for a class of potentials  $V (x) = V_1(x)+V_2(x)$,
 where $ \limsup|x||V_1(x)| <\infty$,  $\limsup V_2(x) =0$  and
$ \limsup |x||\frac{\partial V_2}{\partial r}| <\infty$ ($\frac{\partial V_2}{\partial r}$ is the derivative with radial direction).   We  refer the readers to Simon's  review \cite{simon2017tosio} for the full details on Kato's method, its applications  and related topics.

There are a series of Kumura's and Donnlley's  papers \cite{donnelly1990negative,donnelly1999,kumura1997essential,kumura2010radial,kumuraflat}  studying the eigenvalues embedded into essential spectrum of the Laplacian on manifolds. See Donnlley's review \cite{donnelly2010spectral}. The results are based on construction of energy functions. However they do not fully  use Kato's method. As a result, they need some geometric condition that we have  shown to be unnecessary.

Our  goal  is to develop Kato's method  (also Agmon's and Simon's generalizations) on manifolds.
This is the first of our series papers,  which in particular  implies sharp bounds  for asymptotically hyperbolic manifolds.
Let us consider the eigen-solution $-\Delta u+Vu=\lambda u$.
 We show that under some weak convexity assumption on a manifold,
asymptotical  behaviors of $\Delta r$ and potentials $V$ can determine whether there is an  eigenvalue embedded into essential spectrum,  where $\Delta r$ is the Laplacian of distance function $r(x)$.
$\Delta r$  comes from geometry and $V$ comes from the Sch\"odinger operator on Euclidean space.  The interesting thing is that
$\Delta r$ is the only term involved in the geometry, completely determining the threshold for embedded eigenvalue. For example, even to obtain  our Corollary \ref{Cor1},
Kumura \cite{kumura2010radial} needs an extra condition on Ricci curvature of a manifold.

The original idea of Kato \cite{kato}   to  study the growth of eigen-solutions on Euclidean space  contains four steps:  construct  energy function for eigen-equation;  prove the monotonicity of energy function with respect to $r$ ($r=|x|$ in the Schr\"odinger case); set up the positivity of initial energy; obtain the growth of eigen-solution.
 The first  challenge is to construct the energy function since the simple sum of potential energy and kinetic energy does not work well even for Euclidean space in higher dimensions. We will give the general construction of energy functions on manifolds, which can be adapted to  various  situations easily.
During the process, we  borrow  some derivative estimates from \cite{kumura2010radial} to set up the  monotonicity of energy function (see \S \ref{der}).
However, we  improve the previous  arguments significantly in several aspects,  including the subtle   geometry analysis.
    Another  main novelty here is that we  give   universal method  to treat all cases of manifolds and potentials.
 Here, we just fix some indices to make energy functions work for this paper. We
believe our  method has a  wider applicability. For example, in the following \cite{LiuII}, we give a new way to verify the positivity of initial energy, which works well for the asymptotically flat manifolds.

 The results of \cite{donnelly1979pure,kumura1997essential} show that $ \sigma_{\rm ess}(-\Delta)=[\frac{a^2}{4},\infty)$ if
$\lim \Delta r =a$.

In order to describe  our results, some notations are necessary.
Let $g$ be the metric and  $\nabla$  be  the covariant derivative. Denote
Hessian of $r$ by $\nabla dr$. For simplicity, let $dx$  be volume form (or restriction on sphere if necessary). Sometimes we also use $|\cdot|$ as the norm of vector.
All the functions  $f$ on the manifolds in this paper depend on $x$. For simplicity, we ignore the dependence on sphere and use $f(r)$ instead of $f(x)$.


Let $u $ be a nonzero real solution\footnote{Actually, all the results in the paper hold for  complex solution $u$. We give up  it here because it is not our main goal. } of eigen-equation \eqref{Geigen} and
 define
 \begin{eqnarray*}
   M(r) &=& M(r;u) =(\int_{|r(x)|=r}|u(x)|^2 dx)^{\frac{1}{2}},\\
   N(r) &=& N(r;u)=(\int_{|r(x)|=r}|\frac{\partial u}{\partial r} |^2dx)^{\frac{1}{2}}. 
 \end{eqnarray*}

Our first main result  is
\begin{theorem}{\rm [\textbf{Basic version}]}\label{Mainthm1}
Let  the potential  $V(r)=V_1(r)+V_2(r)$. Assume
\begin{equation*}
    \limsup_{r\to \infty} |rV_1(r)|\leq a_1, \limsup_{r\to \infty}|V_2(r)|=0, \limsup_{r\to \infty}r|\frac{\partial V_2}{\partial r}|\leq a_2,
\end{equation*}
for some non-negative constants $a_1,a_2$.

Assume
\begin{equation*}
   \liminf_{r\to \infty}[  r \nabla dr-a_3\hat{g}]\geq 0,
\end{equation*}
for some $a_3>0$,
where $ \hat{g}=g-dr\otimes dr$, and
\begin{equation*}
 \limsup_{r\to \infty}  r|\Delta r-a_4-\frac{a_5}{r}|\leq  \delta,
\end{equation*}
for some non-negative constants  $ a_4,a_5,\delta$.

Suppose
\begin{equation}\label{Gcons}
    \mu>\delta,2a_3> \mu+\delta, a_3>  1+\delta ,
\end{equation}
and
\begin{equation}\label{Gconl}
   \lambda>\max\{\frac{a_4^2}{4}+\frac{a_2}{\mu- {\delta}}+\frac{1}{4}\frac{(2 {a}_1+{\delta} a_4)^2}{\mu^2-{\delta}^2},\frac{a_4^2}{4}+\frac{a_2}{2(a_3-\delta)}\}.
\end{equation}
Then we have
\begin{equation*}
   \liminf_{r\to \infty} r^{\mu} [M(r)^2+N(r)^2] = \infty .
\end{equation*}


\end{theorem}

Based on  Theorem \ref{Mainthm1} \footnote{By the fact that $\Delta+V$ is essentially selfadjoint, we have $\nabla u\in L^2(M)$
if the eigensolution $u\in L^2(M)$}, we have several immediate  corollaries.
\begin{corollary}\label{Cor1}
Let  the potential  $V(r)=V_1(r)+V_2(r)$. Suppose
\begin{equation*}
    |V_1(r)|=\frac{o(1)}{r}, |V_2(r)|=o(1), |\frac{\partial V_2}{\partial r}|=\frac{o(1)}{r},
\end{equation*}
as $r$ goes to infinity.

Suppose
\begin{equation*}
   \liminf_{r\to \infty}[  r \nabla dr-(1+\epsilon))\hat{g}]\geq 0,
\end{equation*}
for some $\epsilon>0$, and
\begin{equation*}
   \Delta r=a+\frac{b}{r}+ \frac{o(1)}{r}.
\end{equation*}


Then for any $\lambda>\frac{a^2}{4}$ and $\mu>0$, we have
\begin{equation*}
   \liminf_{r\to \infty} r^{\mu} [M(r)^2+N(r)^2] = \infty .
\end{equation*}
In particular, $-\Delta$ admits no   eigenvalue larger than $ \frac{a^2}{4}$.

\end{corollary}
\begin{corollary}\label{Cor2}
Suppose
\begin{equation*}
  \limsup_{r\to \infty}  r|\nabla dr  -\hat{g}|\leq A,
\end{equation*}
and $(n-1)A<1$.
Then $-\Delta$ does not have eigenvalue larger than $
    \frac{(n-1)^2}{4}+\frac{(n-1)^4A^2}{4(1-(n-1)^2A^2)}$.
\end{corollary}
\begin{remark}\label{Redec}
By some direct modifications, a similar result can also be obtained under the assumption
\begin{equation*}
-\frac{B}{r}\leq  \nabla dr  -\hat{g}\leq \frac{A}{r}.
\end{equation*}
as $r$ goes to infinity.
Thus the  corollary  improves Theorem 1.1 in  \cite{kumura2010radial}  by removing extra assumption on the Ricci curvature.
\end{remark}
\begin{corollary}\label{Cor3}
Suppose there exists a   $r_0>0$  such that
\begin{equation*}
    \nabla dr\geq 0
\end{equation*}
for $r=r_0$,
and
\begin{equation*}
      -1-  \frac{2A}{r}\leq K_{\rm rad}(r)\leq-1+ \frac{2A}{r}<0,
\end{equation*}
for $r\geq r_0$, where $K_{\rm rad}(r)$ is the radial curvature \footnote{In Geometry, radial curvature is the sectional curvature with one fixed direction $\frac{\partial}{\partial r}$. We refer the reader to \cite{greene2006function} for formal definition and applications.}.
Suppose
\begin{equation*}
    (n-1)A<1.
\end{equation*}
Then $-\Delta$ does not have eigenvalues larger than $
    \frac{(n-1)^2}{4}+\frac{(n-1)^4A^2}{4(1-(n-1)^2A^2)}$.
    In particular,
    $-\Delta$ does not have eigenvalue larger than $
    \frac{(n-1)^2}{4}$ if $K_{\rm rad}(r)+1=\frac{o(1)}{r}$.
\end{corollary}
Just we mentioned in the introduction,  in Theorem \ref{Mainthm1} $\Delta r$ is from the geometry and $V$ is the potential from Sch\"odinger operator on Euclidean  space.
For the potential part, we  develop  Agmon-Simon's generalization on manifolds. For the geometric part, we  just develop  Kato's method so that there is no derivative of $\Delta r$ involved in.
Our next two theorems  are to  develop Agmon-Simon's generalization on geometric part of manifolds. Namely,
if we know the information of  $\frac{\partial \Delta r}{\partial r}$ (or   gradient of $\Delta r$),  we can get similar results to Theorem \ref{Mainthm1}.

\begin{theorem}\label{Mainthm2} {\rm [\textbf{Gradient version}]}
Let  the potential  $V(r)=V_1(r)+V_2(r)$. Suppose
\begin{equation*}
    \limsup_{r\to \infty} |rV_1(r)|\leq a_1, \limsup_{r\to \infty}|V_2(r)|=0, \limsup_{r\to \infty}r|\frac{\partial V_2}{\partial r}|\leq a_2,
\end{equation*}
for some non-negative constants $a_1,a_2$.

Suppose
\begin{equation*}
   \liminf_{r\to \infty}[  r \nabla dr-a_3\hat{g}]\geq 0,
\end{equation*}
for some $a_3>1$.
Suppose
\begin{equation*}
 \Delta r=a_4+\frac{a_5}{r}+ \frac{\bar{\delta}(r)}{r}
\end{equation*}
and
\begin{equation*}
 \limsup_{r\to \infty}  |\frac{\partial \bar{\delta} (r)}{\partial r}|\leq \delta_1,   \limsup_{r\to \infty}  |(\nabla-\frac{\partial}{\partial r})\bar{\delta} (r)|\leq \delta_2,
 \limsup_{r\to \infty}  |\bar{\delta} (r)|\leq \delta,
\end{equation*}
for some non-negative constants  $ a_4,a_5,\delta_1,\delta_2,\delta$.

Suppose
\begin{equation*}
    2a_3> \mu ,
\end{equation*}
\begin{equation}\label{Gsecondlam2}
  \lambda>  \frac{a_4^2}{4}+\frac{1}{\mu}[a_2+\frac{(2a_1+\delta_1)^2}{4\mu}+\frac{\delta_2^2}{8a_3-4\mu}+\frac{a_4\delta_1}{2}],
\end{equation}
and
\begin{equation}\label{Gsecondlam1}
  \lambda>\frac{a_4^2}{4} + \min_{2\leq s_0\leq 2a_3}\{\frac{a_2}{s_0}+\frac{a_4\delta_1}{2s_0}+\frac{\delta_2^2}{(8 a_3 -4s_0)s_0}\}.
\end{equation}
Then we have
\begin{equation*}
    \liminf_{r\to \infty} r^{\mu} [M(r)^2+N(r)^2] = \infty.
\end{equation*}

\end{theorem}
\begin{remark}
 \begin{itemize}
  \item The bounds  on the right of \eqref{Gsecondlam2} and \eqref{Gsecondlam1} depend on $\delta_1$, not $\delta$.
  \item We can also obtain some interesting corollaries like Corollaries \ref{Cor1} and  \ref{Cor2}. 
 \end{itemize}

\end{remark}
\begin{theorem}\label{Mainthm3}{\rm [\textbf{Mixed version}]}
Let  the potential  $V(r)=V_1(r)+V_2(r)$. Suppose
\begin{equation*}
    \limsup_{r\to \infty} |rV_1(r)|\leq a_1, \limsup_{r\to \infty}|V_2(r)|=0, \limsup_{r\to \infty}r|\frac{\partial V_2}{\partial r}|\leq a_2,
\end{equation*}
for non-negative constants $a_1,a_2$.

Suppose
\begin{equation*}
   \liminf_{r\to \infty}[  r \nabla dr-a_3\hat{g}]\geq 0,
\end{equation*}
for some $a_3>1$.
Suppose
\begin{equation*}
 \Delta r=a_4+\frac{a_5}{r}+ \frac{\bar{\delta}(r)}{r}
\end{equation*}
and
\begin{equation*}
 \limsup_{r\to \infty}  |\frac{\partial \bar{\delta} (r)}{\partial r}|\leq \delta_1,
 \limsup_{r\to \infty}  |\bar{\delta} (r)|\leq \delta,
\end{equation*}
for some non-negative constants  $ a_4,a_5,\delta_1,\delta$.

Suppose
\begin{equation}\label{Gconsthird}
    \mu>\delta,2a_3> \mu+\delta, a_3>  1+\delta ,
\end{equation}
and
\begin{equation}\label{Gconlthird}
   \lambda>\max\{\frac{a_4^2}{4}+\frac{a_2}{\mu-{\delta}}+\frac{a_4 {\delta}_1}{2(\mu-{\delta})}+\frac{ {a}_1^2}{\mu^2- {\delta}^2},\frac{a_4^2}{4}+\frac{2a_2+a_4\delta_1}{4(a_3-\delta)}\}.
\end{equation}
Then we have
\begin{equation*}
    \liminf_{r\to \infty} r^{\mu} [M(r)^2+N(r)^2] = \infty.
\end{equation*}

\end{theorem}

\begin{corollary}\label{Cor4}
Suppose there exists a   $r_0>0$  such that
\begin{equation*}
    \nabla dr\geq 0
\end{equation*}
for $r=r_0$,
and
\begin{equation*}
     -1-  \frac{2A}{r}\leq K_{\rm rad}(r)\leq -1+ \frac{2A}{r}<0,
\end{equation*}
for $r\geq r_0$, where $K_{\rm rad}(r)$ is the radial curvature.
Suppose
\begin{equation*}
    (n-1)A<1.
\end{equation*}
Then $-\Delta$ does not have eigenvalue larger than $
    \frac{(n-1)^2}{4}+\frac{2(n-1)^2A}{1-(n-1)A}$.
\end{corollary}
\begin{remark}
The lower bound  $ \frac{(n-1)^2}{4}+\frac{2(n-1)^2A}{1-(n-1)A}$ is  exactly   the  bound  given in \cite{kumura2010radial}.
\end{remark}
We should remind that  Corollaries \ref{Cor3} and \ref{Cor4} gave different bounds, and   $n$ and $A$ will decide which one is better.
However by   the combination of  energy functions in Theorems \ref{Mainthm2} and \ref{Mainthm3}, we can get a universal  bound, which is better than that in Corollaries \ref{Cor3} and \ref{Cor4}.

\begin{theorem}\label{goodbound}
Suppose there exists a   $r_0>0$  such that
\begin{equation*}
    \nabla dr\geq 0
\end{equation*}
for $r=r_0$,
and
\begin{equation*}
     -1- \frac{2A}{r}\leq K_{\rm rad}(r)\leq -1+\frac{2A}{r}<0,
\end{equation*}
for $r\geq r_0$, where $K_{\rm rad}(r)$ is the radial curvature.
Suppose
\begin{equation*}
    (n-1)A<1
\end{equation*}
and
   \begin{equation*}
    \lambda>\frac{(n-1)^2}{4} +\min_{\sigma\in [0,1]}\{\sigma^2\frac{(n-1)^4A^2}{4(1-(n-1)^2A^2)}+(1-\sigma)\frac{2(n-1)^2A}{1-(n-1)A}\},
   \end{equation*}
   Then $\lambda$ can not be an eigenvalue of the free Laplacian.
\end{theorem}

\begin{remark}
Actually, by   the combination of  energy functions in Theorems \ref{Mainthm2} and \ref{Mainthm3}, we can set up a generally stronger Theorem with more generality.
We do not want to explore the general case here,  we just give a better bound than that in Corollaries \ref{Cor3} and \ref{Cor4}.
\end{remark}
We want to see more about Corollaries \ref{Cor3}, \ref{Cor4} and Theorem \ref{goodbound}.  Radial curvature $K_{\rm rad}(r)$ is a  feature of hyperbolic manifold and flat manifold.

For the asymptotically hyperbolic case,
the sharp transition is given by Kumura   \cite{kumura2010radial} by studying the eigen-solutions directly.  
 He   excludes eigenvalues greater than $\frac{(n-1)^2}{4}$ under the assumption that $K_{\rm rad}(r) + 1 = o(r^{-1})$, and  also  constructs a manifold with the radial curvature  $K_{\rm rad}(r) + 1 = O(r^{-1})$ and with an eigenvalue $\frac{(n-1)^2}{4} + 1$ embedding into its essential spectrum
 $[ \frac{(n-1)^2}{4} , \infty )$. Before that some partial results on the absence of eigenvalues were obtained in  papers \cite{donnelly1990negative,pinsky1979spectrum}.

For the asymptotically flat case, several authors    \cite{kumuraflat,donnelly1992,donnelly1999,escobar1992spectrum} showed the absence of positive eigenvalues of  free Laplacian under some assumptions on the radial curvature. Roughly speaking, they assume  $|K_{\rm rad}(r)|   \leq \frac{\delta}{1+r^2}$ for small $\delta$.
See Donnelly's review paper \cite{donnelly2010spectral} for more results.

Jitomirskaya and Liu also constructed examples which show that dense eigenvalues and singular continuous spectrum can embed  into essential spectrum of Laplacian in both cases \cite{jl1,jl2}. We   mentioned that
Kumura also studied other  related topics in papers \cite{MR2951504,MR2914877,MR3100776}. There are also other topics  about asymptotically  flat   and asymptotically  hyperbolic  manifolds.
See \cite{tao} and the references therein.

The rest of the paper is organized as follows:
in \S 2, we present some basic knowledge. In \S 3, we will give the general construction of energy functions.
In \S 4, by fixing some indices in the energy functions, we   prove   Theorem \ref{Mainthm1} and  Corollaries  \ref{Cor1}, \ref{Cor2}, \ref{Cor3}.
In \S 5, by adapting the energy functions, we  prove   Theorem \ref{Mainthm2}.
In \S 6, by adapting the energy functions, we  prove    Theorems   \ref{Mainthm3}, \ref{goodbound} and Corollary \ref{Cor4}.
Our proof  is self-contained except the unique continuation theorem and some basic geometry results (Lemma \ref{Lekey3}).
\section{Preliminaries and derivative lemma}\label{der}
Let   $S_t=\{x\in M: r(x)=t\}$,  $ \omega\in  S_r$ and $x\in M$. Thus   $(r,\omega)$ is a local coordinate system for $M$ \footnote{$\omega$ depends on $r$. We ignore the dependence for simplicity.}.
Let $<\cdot,\cdot>$ be  the metric on the  Riemannian manifold.  

Choose a function $\rho(r)$, which will be specified later.
 Let  $\hat{L}=e^{\rho}L e^{-\rho}$, where $L=-\Delta+V$. Then
 \begin{align*}
  \begin{CD}
    L^2(M,dg)           @>{ -\Delta + V}>>       L^2(M,dg)          \\
    @V{e^{\rho}}VV                         @VV{e^{\rho}}V  \\
    L^2(M,e^{-2\rho}dg)   @>>{ \hat{ L} }>           L^2(M,e^{-2\rho}dg)
  \end{CD}
\end{align*}

Let
\begin{equation*}
    v=e^{\rho}u.
\end{equation*}
Then, one has
\begin{equation*}
    \nabla u=-\rho^{\prime}e^{-\rho}v \nabla r+e^{-\rho} \nabla v,
\end{equation*}
and
\begin{equation*}
    \Delta u={\rm div} \nabla u=e^{-\rho}\Delta v -2\rho^{\prime}e^{-\rho}\frac{\partial v}{\partial r}+(\rho^{\prime 2}-\rho^{\prime \prime}-\rho^{\prime} \Delta r)e^{-\rho}v.
\end{equation*}

So  the eigen-equation \eqref{Geigen} becomes
\begin{equation}\label{equav}
    -\Delta v+2\rho^{\prime}\frac{\partial v}{\partial r}+(V+V_0)v=\lambda v,
\end{equation}
where
\begin{equation}\label{Gv0}
V_0=\rho^{\prime} \Delta r+\rho^{\prime \prime}-\rho^{\prime 2}.
\end{equation}

\begin{lemma}\label{Keylemma1}\cite{kumura2010radial}
Let $X$ be a vector field. Then
\begin{equation}\label{Gder1}
    \frac{\partial}{\partial r}\int_{S_r}<X,\nabla r> e^{-2\rho}dx=\int_{S_r}({\rm div} X -2\rho^{\prime}<X,\nabla r>) e^{-2\rho}dx.
\end{equation}

\end{lemma}
\begin{proof}
First, one has
\begin{equation*}
    {\rm div} (X e^{-2\rho})=e^{-2\rho} {\rm div} (X )-2\rho^{\prime} e^{-2\rho}<X,\nabla r>
\end{equation*}
Integration by part, we get
\begin{equation*}
   \int_{S_{t_2}}<X,\nabla r> e^{-2\rho}dx- \int_{S_{t_1}}<X,\nabla r> e^{-2\rho}dx=\int_{t_1\leq|r(x)|\leq t_2}({\rm div} X -2\rho^{\prime}<X,\nabla r>) e^{-2\rho}dx,
\end{equation*}
which implies \eqref{Gder1}. 
\end{proof}

\begin{lemma}\label{Keylemma2}\cite{kumura2010radial}
\begin{equation*}
    \frac{\partial}{\partial r}\int_{S_r}f  e^{-2\rho}dx=\int_{S_r} [\frac{\partial f}{\partial r}+f(\Delta r-2\rho^{\prime})]e^{-2\rho}dx.
\end{equation*}
\end{lemma}
\begin{proof}
Let $X= f \nabla r$.  By direct computation, one has
\begin{equation*}
    {\rm div} X= {\rm div}(f \nabla r)= \frac{\partial f}{\partial r}  +f\Delta r.
\end{equation*}
Putting $X$ into   Lemma \ref{Keylemma1},  Lemma \ref{Keylemma2} follows.
\end{proof}
Now we always assume $u$ is a nonzero solution of $-\Delta u+Vu=\lambda u$, where $V=V_1+V_2$.
Let
\begin{equation*}
    v_m=r^m v
\end{equation*}
with $m\geq 0$.
By \eqref{equav}, we get the equation of $v_m$,
\begin{equation}\label{equav1}
    \Delta v_m -(\frac{2m}{r}+2\rho^{\prime} )\frac{\partial v_m}{\partial r}+(\frac{m(m+1)}{r^2}+\frac{m}{r}(2\rho^{\prime}-\Delta r)-V_0-V_1-V_2+\lambda)v_m=0.
\end{equation}
\section{Construction of the energy functions}

 In this section, we will give the general construction of energy functions and derive the formulas for their derivatives.

 Let $A_r$ be the Laplacian on sphere $r$. Using $\Delta u=\frac{\partial^2 u}{\partial r^2}+\Delta r \frac{\partial u}{\partial r}+A_r u$,
  \eqref{equav1} becomes
 \begin{equation}\label{equav1apr416}
    \frac{ \partial^2v_m }{\partial r^2}-( -\Delta r+\frac{2m}{r}+2\rho^{\prime} )\frac{\partial v_m}{\partial r}+A_r v_m+(\frac{m(m+1)}{r^2}+\frac{m}{r}(2\rho^{\prime}-\Delta r)-V_0-V_1-V_2+\lambda)v_m=0.
\end{equation}
Let us mention our intuition to construct energy functions. We view \eqref{equav1apr416} as one dimensional Schr\"odinger operator ($r$ is the variable).
 \eqref{equav1apr416} is not the normal form-$-D^2+q$ since $-\Delta r+\frac{2m}{r}+2\rho^{\prime}$   is not 0.
So the first step we need to do is to choose $\rho$ such that $-\Delta r+\frac{2m}{r}+2\rho^{\prime}$ is smaller than   $\frac{O(1)}{r}$.
The energy function for equation $u^{\prime\prime}+qu=0$ is $\frac{1}{2}|u^{\prime}|^2+\frac{1}{2}qu^2$.
Similarly, the usual energy functions of \eqref{equav1apr416} are taken with the form as
$\frac{1}{2}|\frac{\partial v_m}{\partial r}|^2+\frac{1}{2}qv_m^2$ with average on the sphere. Since we can not make $-\Delta r+\frac{2m}{r}+2\rho^{\prime}$  zero,  extra term $q_1\frac{\partial v_m}{\partial r} v_m$ should be added into the energy functions. By the fact
\begin{equation*}
  \int_{S_r}[\frac{1}{2}|\frac{\partial v_m}{\partial r}|^2 +\frac{1}{2}(A_r v_m,v_m)] dx=\int_{S_r}[|\frac{\partial v_m}{\partial r}|^2-\frac{1}{2}|\nabla v_m|^2]dx,
\end{equation*}
it is natural to construct
\begin{eqnarray*}
  F(m,r,t,s) &=& r^s\int_{S_r}\frac{1}{2}[q_1\frac{\partial v_m}{\partial r} v_m+(\frac{m(m+1)}{r^2}-\frac{t}{r}+q_2+\lambda)v_m^2]e^{-2\rho}dx\\
  &&+r^s\int_{S_r}[|\frac{\partial v_m}{\partial r}|^2-\frac{1}{2}|\nabla v_m|^2]e^{-2\rho}dx \\
   &=& {\rm I} +{\rm II}+{\rm III},
\end{eqnarray*}
where
\begin{eqnarray*}
  {\rm I} &=& r^s\int_{S_r}[|\frac{\partial v_m}{\partial r}|^2-\frac{1}{2}|\nabla v_m|^2]e^{-2\rho}dx \\
   &=& \frac{1}{2} r^s\int_{S_r}[|\frac{\partial v_m}{\partial r}|^2-|\nabla_{\omega} v_m|^2]e^{-2\rho}dx ,
\end{eqnarray*}
and
\begin{equation*}
   {\rm II}= \frac{1}{2}r^s\int_{S_r}[\frac{m(m+1)}{r^2}-\frac{t}{r}+q_2+\lambda]v_m^2 e^{-2\rho}dx,
\end{equation*}
and
\begin{equation*}
   {\rm III}= \frac{1}{2}r^s\int_{S_r} [q_1\frac{\partial v_m}{\partial r} v_m]e^{-2\rho}dx.
\end{equation*}
We begin with the derivation of $\frac{\partial }{\partial r}{\rm I}$.

By Lemma \ref{Keylemma2},  one has
\begin{eqnarray*}
  \frac{\partial }{\partial r}{\rm I} &=& sr^{s-1}\int_{S_r }[\frac{1}{2}|\frac{\partial v_m}{\partial r}|^2-\frac{1}{2}|\nabla_{\omega} v_m|^2]e^{-2\rho}dx+r^s\int_{S_r}[\frac{\partial v_m}{\partial r}\frac{\partial^2v_m}{\partial r^2}-\frac{1}{2} \frac{\partial }{\partial r}<\nabla_{\omega} v_m,\nabla_{\omega}v_m>]e^{-2\rho} dx\\
   && +r^s\int_{S_r}  \frac{1}{2}(\Delta r-2\rho^{\prime})[|\frac{\partial v_m}{\partial r}|^2-|\nabla_{\omega} v_m|^2]e^{-2\rho}dx.
\end{eqnarray*}
Using $ \Delta v_m=\frac{\partial^2v_m}{\partial r^2}+\Delta r\frac{\partial v_m}{\partial r}+\Delta_{\omega} v_m$, we get
\begin{eqnarray*}
  \frac{\partial }{\partial r}{\rm I} &=& \int_{S_r}[\frac{s}{2}r^{s-1}|\frac{\partial v_m}{\partial r}|^2-2r^s\rho ^{\prime}|\frac{\partial v_m}{\partial r}|^2+r^s\frac{\partial v_m}{\partial r} \Delta v_m+ \frac{r^s}{2}(2\rho^{\prime}-\Delta r)|\frac{\partial v_m}{\partial r}|^2 ]e^{-2\rho}dx\\
   &&+
   r^{s}\int_{S_r}[(  -\frac{s}{2r }+\frac{1}{2}(2\rho^{\prime}-\Delta r))\hat{g}(\nabla v_m,\nabla v_m)]e^{-2\rho}dx\\
   &&+r^s\int_{S_r}[<\nabla_{\omega} \frac{\partial v_m}{\partial r},\nabla_{\omega}v_m>-\frac{1}{2} \frac{\partial }{\partial r}<\nabla_{\omega} v_m,\nabla_{\omega}v_m>]e^{-2\rho} dx.
\end{eqnarray*}
By some basic computation, one has
\begin{eqnarray*}
  <\nabla_{\omega} \frac{\partial v_m}{\partial r},\nabla_{\omega}v_m>-\frac{1}{2} \frac{\partial }{\partial r}<\nabla_{\omega} v_m,\nabla_{\omega}v_m> &=& <\nabla_{\omega} \frac{\partial v_m}{\partial r},\nabla_{\omega}v_m>- <\nabla _{ \frac{\partial }{\partial r}}\nabla_{\omega} v_m,\nabla_{\omega}v_m> \\
   &=& (\nabla dr)(\nabla_{\omega}v_m,\nabla_{\omega}v_m) .
\end{eqnarray*}
Finally we get
\begin{eqnarray}
 \nonumber \frac{\partial }{\partial r}{\rm I} &=& \int_{S_r}[\frac{s}{2}r^{s-1}|\frac{\partial v_m}{\partial r}|^2-2r^s\rho ^{\prime}|\frac{\partial v_m}{\partial r}|^2+r^s\frac{\partial v_m}{\partial r} \Delta v_m+ \frac{r^s}{2}(2\rho^{\prime}-\Delta r)|\frac{\partial v_m}{\partial r}|^2 ]e^{-2\rho}dx\\
   &&+
   r^{s}\int_{S_r}[( \nabla dr +(-\frac{s}{2r }+\frac{1}{2}(2\rho^{\prime}-\Delta r))\hat{g})(\nabla v_m,\nabla v_m)]e^{-2\rho}dx.\label{GpartialI}
\end{eqnarray}
Now we are in the position to obtain $ \frac{\partial}{\partial r}{\rm II}$.

By Lemma \ref{Keylemma2} again, one has
\begin{eqnarray}
  \nonumber  \frac{\partial}{\partial r}{\rm II}&= &\int_{S_r}  [\frac{\partial}{\partial r} \frac{r^s}{2}(\frac{m(m+1)}{r^2}-\frac{t}{r}+q_2+\lambda)v_m^2]e^{-2\rho}dx \\
 \nonumber   &&+\int_{S_r}(\Delta r-2\rho^{\prime})\frac{r^s}{2}[\frac{m(m+1)}{r^2}-\frac{t}{r}+q_2+\lambda]v_m^2e^{-2\rho}dx\\
 \nonumber   &=& \int_{S_r}  [ \frac{s-2}{2} r^{s-3}m(m+1)-\frac{s-1}{2}t r^{s-2}+\frac{r^s}{2}\frac{\partial q_2}{\partial r}+\frac{s}{2}r^{s-1}q_2+\lambda\frac{s}{2}r^{s-1}]v_m^2 e^{-2\rho}dx\\
 \nonumber   &&+r^s \int_{S_r}[ \frac{m(m+1)}{r^2}-\frac{t}{r}+q_2+\lambda]v_m\frac{\partial v_m}{\partial r} e^{-2\rho}dx\\
   &&+\frac{r^s}{2}\int_{S_ r}(\Delta r-2\rho^{\prime})[\frac{m(m+1)}{r^2}-\frac{t}{r}+q_2+\lambda]v_m^2e^{-2\rho}dx.\label{GpartialII}
\end{eqnarray}
Similarly,
by Lemma \ref{Keylemma2} again, we have
\begin{eqnarray}
 \nonumber\frac{\partial}{\partial r}{\rm III}&= &\int_{S_r} [\frac{1}{2} q_1r^s  v_m\frac{\partial v_m}{\partial r}(\Delta r-2\rho^{\prime})+\frac{\partial}{\partial r}(r^s \frac{1}{2} q_1\frac{\partial v_m}{\partial r}v_m) ]e^{-2\rho}dx\\
\nonumber   &=& \int_{S_ r} [\frac{r^s}{2} q_1  (\Delta r-2\rho^{\prime})+\frac{1}{2}sr^{s-1} q_1+\frac{1}{2}r^s\frac{\partial q_1}{\partial r}]\frac{\partial v_m}{\partial r}v_m]e^{-2\rho}dx\\
 \nonumber  &&+\int_{S_ r} [\frac{1}{2}r^s q_1|\frac{\partial v_m}{\partial r}|^2+\frac{1}{2}r^s q_1v_m\frac{\partial ^2v_m}{\partial^2 r}]e^{-2\rho}dx\\
\nonumber   &=&\int_{S_r} [\frac{r^s}{2} q_1  (\Delta r-2\rho^{\prime})+\frac{s}{2} r^{s-1}q_1+\frac{r^s}{2}\frac{\partial q_1}{\partial r}]\frac{\partial v_m}{\partial r}v_me^{-2\rho}dx\\
 \nonumber  &&+\int_{S_ r} [\frac{r^s}{2} q_1|\frac{\partial v_m}{\partial r}|^2+\frac{r^s}{2} q_1v_m(\Delta v_m-\Delta r \frac{\partial v_m}{\partial r}-\Delta_{\omega}v_m)]e^{-2\rho}dx.\label{GpartialIII}\\
\end{eqnarray}
Putting \eqref{GpartialI},\eqref{GpartialII}, \eqref{GpartialIII} together and using \eqref{equav1}, we conclude that
\begin{equation}\label{Gpartial}
     \frac{\partial F(m,r,t,s)}{\partial r} =\frac{\partial{\rm I}}{\partial r} +
     \frac{\partial {\rm II}}{\partial r}+\frac{\partial{\rm III}}{\partial r}\;\;\;\;\;\;\;\;\;\;\;\;\;\;\;\;\;\;\;\;\;\;\;\;\;\;\;\;\;\;\;\;\;\;\;\;\;\;\;\;\;\;\;\;\;\;\;\;\;\;\;\;\;\;\;\;\;\;\;\;
\end{equation}

\begin{eqnarray}
  \nonumber \;\;\;\;\;\;\;\;\; \;\;\;\;\;\;\;\;\;\;\;\;\;\;\;\;\;\;\;\;\;\;\;\;\;\;\;&=&\int_{S_r} [r^{s}( (\nabla dr) -(\frac{s}{2r }-\frac{1}{2}(2\rho^{\prime}-\Delta r))\hat{g})(\nabla v_m,\nabla v_m)+\frac{1}{2} q_1r^sv_m(-\Delta_{\omega}v_m)]e^{-2\rho} dx\\
   \nonumber &&+\int_{S_r} [2mr^{s-1} +\frac{r^s}{2}(2\rho^{\prime}-\Delta r)+\frac{1}{2} q_1r^s+\frac{s}{2}r^{s-1}]|\frac{\partial v_m}{\partial r}|^2 e^{-2\rho}dx\\
   \nonumber &&+\int_{S_r}[r^s(V_0+V_1+V_2+q_2-\frac{t}{r})+r^{s-1}m(\Delta r-2\rho^{\prime})]\frac{\partial v_m}{\partial r}v_m  e^{-2\rho}dx\\
 \nonumber   && + \int_{S_r} [\frac{1}{2}sr^{s-1} q_1+mr^{s-1} q_1+\frac{1}{2}r^s\frac{\partial q_1}{\partial r}]\frac{\partial v_m}{\partial r}v_m e^{-2\rho}dx\\
  \nonumber  && +\int_{S_r} [\frac{s-2}{2} r^{s-3}m(m+1)-\frac{s-1}{2}t r^{s-2}+\frac{r^s}{2}\frac{\partial q_2}{\partial r}+\frac{s}{2}r^{s-1}q_2+\lambda\frac{s}{2}r^{s-1}]v_m^2e^{-2\rho}dx\\
  \nonumber  &&+\int_{S_r}\frac{r^s}{2}(\Delta r-2\rho^{\prime})[\frac{m(m+1)}{r^2}-\frac{t}{r}+q_2+\lambda]v_m^2e^{-2\rho}dx\\
   \nonumber  && +\int_{S_r}-\frac{r^s}{2}  q_1[\frac{m(m+1)}{r^2}+\frac{m}{r}(2\rho^{\prime}-\Delta r)-V_0-V_1-V_2+\lambda] v_m^2e^{-2\rho}dx
\end{eqnarray}

\section{Proof of   Theorem \ref{Mainthm1} and some Corollaries }\label{proof1}
Let
\begin{eqnarray*}
  \bar{\delta} (r) &=& r(\Delta r-a_4-\frac{a_5}{r}) \\
  \bar{a}_1(r) &=& rV_1(r)\\
   \bar{a}_2(r) &=& r\frac{\partial V_2}{\partial r}.
\end{eqnarray*}

By the  assumptions of  Theorem \ref{Mainthm1},
\begin{equation*}
    \limsup_{r\to \infty} |\bar{\delta} (r)|\leq \delta, \limsup_{r\to \infty} |\bar{a}_1 (r)|\leq a_1,\limsup_{r\to \infty} |\bar{a}_2 (r)|\leq a_2.
\end{equation*}

Let $0<t<1$  be small enough,  $2\rho^{\prime}=a_4+\frac{a_5}{r}$  and $q_1=0$. Direct computation  of \eqref{Gv0} implies that
\begin{equation*}
V_0=\frac{a_4^2}{4}+\frac{a_4a_5}{2r}+\frac{a_4\bar{\delta}}{2r}+\frac{O(1)}{r^2}.
\end{equation*}
We should mention that $O(1)$ and $o(1)$  only depend  on constants in the  assumptions of  Theorem \ref{Mainthm1}, not depend  on $m,t$.

Let $q_2=-\frac{a_4^2}{4}-\frac{a_4a_5}{2r}-V_2$.  By \eqref{Gpartial}, we have
\begin{eqnarray}
  \label{Gpartialcase1} \frac{\partial F(m,r,t,s)}{\partial r} 
   &=& r^{s-1}\int_{S_r}[( r(\nabla dr) -(\frac{s}{2 }+\frac{1}{2}\bar{\delta})\hat{g})(\nabla v_m,\nabla v_m)]e^{-2\rho}dx\\
 \label{Gpartialcase2}  &&+r^{s-1}\int_{S_r}[2m -\frac{\bar{\delta}}{2}+\frac{s}{2}]|\frac{\partial v_m}{\partial r}|^2e^{-2\rho}dx\\
  \label{Gpartialcase3} &&+r^{s-1}\int_{S_r}[\bar{a}_1+\frac{\bar{\delta} a_4}{2}-t+\bar{\delta}\frac{m} {r}+o(1)]\frac{\partial v_m}{\partial r}v_me^{-2\rho}dx\\
  \label{Gpartialcase4} && +r^{s-1}\int_{S_r}[(\lambda-\frac{a_4^2}{4})(\frac{s}{2}+\frac{\bar{\delta}}{2})-\frac{\bar{a}_2}{2}+o(1)]v_m^2e^{-2\rho}dx\\
   \label{Gpartialcase5}&&+r^{s-1}\int_{S_r} \frac{m(m+1)}{r^2}[ \frac{s-2}{2}+\frac{\bar{\delta}}{2}]v_m^2 e^{-2\rho}dx
\end{eqnarray}

\begin{theorem}\label{Thmpositive}
Under the assumptions  of Theorem \ref{Mainthm1},
 there exist $s_0,R_0,m_0>0$ such that for $m\geq m_0$ and $r\geq R_0$,
 \begin{equation*}
    \frac{\partial F(m,r,t,s_0)}{\partial r} >0.
 \end{equation*}
\end{theorem}
\begin{proof}
Let  $s_0$ be such that  $s_0<2a_3-\delta$ and sufficiently close to $2a_3-\delta $.
By the assumption $a_3>1+\delta$ (see \eqref{Gcons}), one has
\begin{equation*}
     r(\nabla dr) -(\frac{s_0}{2 }+\frac{1}{2}\bar{\delta})\hat{g} \geq 0,
\end{equation*}
for large $r$, which implies
\begin{equation}\label{Gpartial2positive}
    \eqref{Gpartialcase1}>0.
\end{equation}
By assumption $a_3>1+\delta$ (see \eqref{Gcons}) again and $\lambda>\frac{a_4^2}{4}+\frac{a_2}{2(a_3-\delta)}$ (see \eqref{Gconl}) , one has
\begin{equation*}
    (\lambda-\frac{a_4^2}{4})(\frac{s_0}{2}+\frac{\bar{\delta}}{2})-\frac{\bar{a}_2}{2}>0,
\end{equation*}
and
\begin{equation*}
\frac{s_0-2}{2}+\frac{\bar{\delta}}{2}>0.
\end{equation*}

By Cauchy-Schwartz inequality, one has
\begin{equation*}
     [\frac{m(m+1)}{r^2} (\frac{s_0-2}{2}+\frac{\bar{\delta}}{2})+(\lambda-\frac{a_4^2}{4})(\frac{s_0}{2}+\frac{\bar{\delta}}{2})-\frac{\bar{a}_2}{2}]v_m^2+[2m -\frac{\bar{\delta}}{2}+\frac{s_0}{2}]|\frac{\partial v_m}{\partial r}|^2>|[\frac{\bar{a}_1}{2}+\frac{\bar{\delta} a_4}{2}-t+\bar{\delta}\frac{m} {r}]\frac{\partial v_m}{\partial r}v_m|
\end{equation*}
 for large $m$ and $r$. Thus, one has
 \begin{equation}\label{Gpartial5positive}
    |\eqref{Gpartialcase3}|<\eqref{Gpartialcase2}+\eqref{Gpartialcase4}+\eqref{Gpartialcase5}
 \end{equation}
for large $m$ and $r$.

By \eqref{Gpartial2positive} and \eqref{Gpartial5positive}, we  obtain Theorem \ref{Thmpositive}.
\end{proof}
\begin{theorem}\label{Thmpositivezero}
Let  $s<\mu$ and $s$ be sufficiently close to $\mu$. Then
under the conditions of Theorem \ref{Mainthm1},
 we have
 \begin{equation*}
    \frac{\partial F(0,r,0,s)}{\partial r} >0,
 \end{equation*}
 for large $r$.
\end{theorem}
\begin{proof}
Let $m=0, t=0$ in \eqref{Gpartialcase1}-\eqref{Gpartialcase5}, one has
\begin{eqnarray*}
   \frac{\partial F(0,r,0,s)}{\partial r}
   &=&\int_{S_r} r^{s-1}[ r(\nabla dr) -(\frac{s}{2 }+\frac{1}{2}\bar{\delta})\hat{g})(\nabla v,\nabla v)]e^{-2\rho}dx\\
   &&+r^{s-1}\int_{S_r}[ -\frac{\bar{\delta}}{2}+\frac{s}{2}]|\frac{\partial v}{\partial r}|^2e^{-2\rho}dx\\
   &&+r^{s-1}\int_{S_r}[\bar{a}_1+\frac{\bar{\delta} a_4}{2}+o(1)]\frac{\partial v}{\partial r}v e^{-2\rho}dx\\
   && +r^{s-1}\int_{S_r}[(\lambda-\frac{a_4^2}{4})(\frac{s}{2}+\frac{\bar{\delta}}{2})-\frac{\bar{a}_2}{2}+o(1)]v^2 e^{-2\rho}dx.
\end{eqnarray*}
We will show  that for large $r$,
\begin{equation*}
    \frac{\partial F(0,r,0,s)}{\partial r}>0.
\end{equation*}
By Cauchy Schwartz inequality, it suffices to prove
\begin{equation}\label{Gsolveeq}
    4[(\lambda-\frac{a_4^2}{4})(\frac{s}{2}+\frac{\bar{\delta}}{2})-\frac{\bar{a}_2}{2}][-\frac{\bar{\delta}}{2}+\frac{s}{2}]>
    |\bar{a}_1+\frac{\bar{\delta}}{2}a_4|^2.
\end{equation}
Solving  inequality \eqref{Gsolveeq},
we get
\begin{equation}\label{Gsolveeq1}
     \lambda>\frac{a_4^2}{4}+\frac{a_2}{s+\bar{\delta}}+\frac{1}{4}\frac{(2\bar{a}_1+\bar{\delta} a_4)^2}{s^2-\bar{\delta}^2}.
\end{equation}
It is clear that \eqref{Gsolveeq1} holds if
\begin{equation}\label{Gsolveeq2}
     \lambda>\frac{a_4^2}{4}+\frac{a_2}{s- {\delta}}+\frac{1}{4}\frac{( 2{a}_1+{\delta} a_4)^2}{s^2-{\delta}^2},
\end{equation}
which follows from the  assumption \eqref{Gconl} and the fact $s$ is close to $\mu$.
\end{proof}
\begin{theorem}\label{Keythm2}
There exist $m_0\geq 0$ and $R_0>0$ such that
\begin{equation*}
    F(m_0,r,t,0)>0
\end{equation*}
for all $r\geq R_0$.
\end{theorem}
\begin{proof}
It is easy  to check that
\begin{eqnarray*}
  F(m,r,t,0) &=& \int_{S_r}[\frac{1}{2}(\frac{m(m+1)}{r^2}-\frac{t}{r}+q_2+\lambda)v_m^2 +(|\frac{\partial v_m}{\partial r}|^2-\frac{1}{2}|\nabla v_m|^2)]e^{-2\rho}dx\\
   &=& \frac{r^{2m}}{2}\int_{S_r} [(\frac{2m^2+m}{r^2}-\frac{t}{r}+q_2+\lambda)v^2+|\frac{\partial v}{\partial r}|^2 +2\frac{m}{r}\frac{\partial v}{\partial r}v-|\nabla_{\omega} v|^2]e^{-2\rho}dx.
\end{eqnarray*}
By unique continuation theorem and fact that  $u$ is nonzero, there exists large enough    $R_0$ such that   $ \int_{S_{R_0}} v^2 e^{-2\rho}dx \neq 0$. Let $m_0$ be large enough so that
\begin{equation*}
    F(m_0,R_0,t,0)>0.
\end{equation*}
By Theorem \ref{Thmpositive}, we get
\begin{equation*}
    F(m_0,r,t,0)>0,
\end{equation*}
for all $r>R_0$.

\end{proof}
\begin{theorem}\label{Keythm3}
Assume that  $ \int_{S_r} v^2 e^{-2\rho}dx $ is not monotone increasing (with respect to  $r$) in any semi-infinite interval $r\geq R$.
Then there exists a sequence $r_n$ goes to infinity such that
\begin{equation*}
    F(0,0,r_n,0)>0,
\end{equation*}
\end{theorem}
\begin{proof}
By the assumption, there exists a sequence $r_n$ goes to infinity and
such that
\begin{equation*}
     \frac{\partial}{\partial r}\int_{S_r} v^2 e^{-2\rho}dx <0
\end{equation*}
for  $r=r_n$.
By Lemma \ref{Keylemma2}, one has
\begin{equation}\label{Gdecrease}
    \int_{S_r} [2v\frac{\partial v}{\partial r} +(\Delta r-2\rho^{\prime})v^2]e^{-2\rho}dx <0
\end{equation}
for   $r=r_n$.
By some direct computation, we have
\begin{eqnarray}
\nonumber  F(m_0,r,t,0) &=& \int_{S_{r}}[\frac{1}{2}(\frac{m_0(m_0+1)}{r^2}-\frac{t}{r}+q_2+\lambda)v_{m_0}^2+(|\frac{\partial v_{m_0}}{\partial r}|^2-\frac{1}{2}|\nabla v_{m_0}|^2)]e^{-2\rho}dx\\
 \nonumber  &=& \frac{r^{2m_0}}{2}\int_{S_{r}} [(\frac{m_0(m_0+1)}{r^2}-\frac{t}{r}+q_2+\lambda)v^2+|\frac{\partial v}{\partial r}+\frac{m_0}{r}v|^2-|\nabla_{\omega} v|^2]e^{-2\rho}dx\\
  \nonumber &=& \frac{r^{2m_0}}{2}\int_{S_{r}} [(q_2+\lambda)v^2+|\frac{\partial v}{\partial r}|^2-|\nabla_{\omega} v|^2]e^{-2\rho}dx\\
  \nonumber && +\frac{r^{2m_0}}{2}\int_{S_{r}} [(\frac{2m_0^2+m_0}{r^2}-\frac{t}{r})v^2+2\frac{m_0}{r}\frac{\partial v}{\partial r}v]e^{-2\rho}dx\\
   \nonumber &=& \frac{r^{2m_0}}{2}\int_{S_{r}} [(q_2+\lambda)v^2+|\frac{\partial v}{\partial r}|^2-|\nabla_{\omega} v|^2]e^{-2\rho}dx\\
  \nonumber && +\frac{m_0r^{2m_0}}{2r}\int_{S_{r}} [2v\frac{\partial v}{\partial r} +(\Delta r-2\rho^{\prime})v^2]e^{-2\rho}dx\\
 \nonumber  && +\frac{r^{2m_0}}{2}\int_{S_{r}} [\frac{2m_0^2+m_0}{r^2}-\frac{t}{r}-\frac{m_0}{r}(\Delta r-2\rho^{\prime}) ]v^2e^{-2\rho}dx\\
 &=& r^{2m_0}F(0,0,r,0)\label{Gdecrease1}\\
  \nonumber && +\frac{m_0r^{2m_0}}{2r}\int_{S_{r}} [2v\frac{\partial v}{\partial r} +(\Delta r-2\rho^{\prime})v^2]e^{-2\rho}dx\\
 \nonumber  && +\frac{r^{2m_0}}{2}\int_{S_{r}} [\frac{2m_0^2+m_0}{r^2}-\frac{t}{r}-\frac{m_0}{r}(\Delta r-2\rho^{\prime}) ]v^2e^{-2\rho}dx.
\end{eqnarray}
For large $r$, one has
\begin{equation*}
    \frac{r^{2m_0}}{2}\int_{S_ R} [\frac{2m_0^2+m_0}{r^2}-\frac{t}{r}-\frac{m_0}{r}(\Delta r-2\rho^{\prime}) ]v^2e^{-2\rho}dx<0,
\end{equation*}
since $t>0$.

Combing with \eqref{Gdecrease} and \eqref{Gdecrease1},
one has
\begin{equation*}
    F(m_0,r_n,t,0)<F(0,r_n,0,0).
\end{equation*}

By  Theorem \ref{Keythm2}, we have $ F(m_0,r_n,t,0) >0$. Thus
we get
\begin{equation*}
  F(0,0,r_n,0) >0.
\end{equation*}
\end{proof}
{\textbf{Proof of Theorem \ref{Mainthm1}}.}
\begin{proof}
It suffices to assume  $ \int_{S_r} v^2 e^{-2\rho}dx $ is not monotone increasing (with respect to  $r$) in any semi-infinite interval $r\geq R$.

By Theorems  \ref{Thmpositivezero} and \ref{Keythm3}, there exits $\gamma>0$ such that
\begin{equation*}
    F(0,r,0,s)\geq\gamma,
\end{equation*}
for large $r$.
Thus
\begin{equation*}
    r^s\int_{S_r}(v^2+|\frac{\partial v}{\partial r}|^2)e^{-2\rho}dx\geq \gamma_0
\end{equation*}
for some $\gamma_0>0$.

By the fact  $v=e^{\rho}u$, we get that
\begin{equation*}
   \liminf_{r\to \infty} r^{s} [M(r)^2+N(r)^2]>0 .
\end{equation*}
Recalling that   $s<\mu$,  we have
\begin{equation*}
   \liminf_{r\to \infty} r^{\mu} [M(r)^2+N(r)^2]=\infty .
\end{equation*}
\end{proof}

{\textbf{Proof of Corollary \ref{Cor1}}}
\begin{proof}
The  Corollary \ref{Cor1}  follows from Theorem \ref{Mainthm1} directly.
\end{proof}
{\textbf{Proof of Corollary \ref{Cor2}}}
\begin{proof}

The  Corollary \ref{Cor2}  follows from Theorem \ref{Mainthm1} and the fact that $\Delta r$ is the trace of $\nabla dr$.
\end{proof}
Before we finish the proof of  Corollary \ref{Cor3}, a lemma is necessary.

\begin{lemma}\label{Lekey3}
Suppose there exists a   $r_0>0$  such that
\begin{equation*}
    \nabla dr\geq 0
\end{equation*}
for $r=r_0$,
and
\begin{equation*}
     -1-  \frac{2A}{r}\leq K_{\rm rad}(r)\leq-1+ \frac{2A}{r}<0,
\end{equation*}
for $r\geq r_0$, where $K_{\rm rad}(r)$ is the radial curvature.

Then
\begin{equation}\label{Gequ4new}
       |\nabla dr  -\hat{g}|\leq \frac{A+o(1)}{r},
\end{equation}
and
\begin{equation}\label{Gequ5new}
       |\frac{\partial \Delta r}{\partial r} |\leq \frac{4(n-1)A+o(1)}{r}.
\end{equation}

\end{lemma}
\begin{proof}
\eqref{Gequ4new} and \eqref{Gequ5new}  can be proved by comparison  theorem and   Weitzenb\"{o}ck formula. See \cite{kumura2010radial} for details.
\end{proof}
{\textbf{Proof of Corollary \ref{Cor3}}}
\begin{proof}
Under the curvature condition of Corollary \ref{Cor3} and  Lemma \ref{Lekey3}, one has
\begin{equation*}
  \limsup_{r\to \infty}  |\nabla dr  -\hat{g}|\leq \frac{A}{r}.
\end{equation*}
Now  Corollary \ref{Cor3}  follows from Corollary \ref{Cor2} .
\end{proof}
\section{Proof of Theorem \ref{Mainthm2} }

 Let $q_1=\Delta r-2\rho^{\prime}$, $q_2=-\frac{a_4^2}{4}-\frac{a_4\bar{\delta}}{2r}-\frac{a_4a_5}{2r}-V_2$ and the others as the same in \S \ref{proof1}.
 By \eqref{Gpartial} and integration by part, one has
\begin{eqnarray*}
  \frac{\partial F(m,r,t,s)}{\partial r}
   &=& r^{s-1}\int_{S_r}[ (r(\nabla dr) -\frac{s}{2 }\hat{g})(\nabla v_m,\nabla v_m)+\frac{1}{2}r^s v_m <\nabla _{\omega}q_1 ,\nabla _{\omega}v_m>]e^{-2\rho}dx\\
   &&+r^{s-1}\int_{S_r}[2m +\frac{s}{2}]|\frac{\partial v_m}{\partial r}|^2e^{-2\rho}dx\\
   &&+r^{s-1}\int_{S_r}r[q_2-\frac{t}{r}+V_0+V_1+V_2]\frac{\partial v_m}{\partial r}v_m e^{-2\rho}dx\\
   && +r^{s-1}\int_{S_r} [\frac{s}{2}  q_1+\frac{1}{2}r\frac{\partial q_1}{\partial r}+2m q_1]\frac{\partial v_m}{\partial r}v_me^{-2\rho}dx\\
   && +\int_{S_r}[\frac{s-2}{2} r^{s-3}m(m+1)-\frac{s-1}{2}t r^{s-2}+\frac{r^s}{2}\frac{\partial q_2}{\partial r}+\frac{s}{2}r^{s-1}q_2+\lambda\frac{s}{2}r^{s-1}]v_m^2e^{-2\rho}dx\\
   &&+\int_{S_r}\frac{1}{2}r^sq_1[-\frac{t}{r}+\frac{m}{r}q_1+q_2+V_0+V_1+V_2]v_m^2e^{-2\rho}dx.
\end{eqnarray*}
Let
\begin{equation*}
\frac{\partial \bar{\delta} (r)}{\partial r}=\bar{ \delta}_1,   (\nabla-\frac{\partial}{\partial r})\bar{\delta} (r)= \bar{\delta}_2.
\end{equation*}
By the assumptions of Theorem \ref{Mainthm2}, we get
\begin{eqnarray}
 \nonumber \frac{\partial F(m,r,t,s)}{\partial r}
\nonumber   &=& r^{s-1}\int_{S_r}[ (r(\nabla dr) -(\frac{s}{2 })\hat{g})(\nabla v_m,\nabla v_m)+\frac{1}{2}r^{s-1} v_m <\bar{\delta}_2 ,\nabla _{\omega}v_m>]e^{-2\rho}dx\\
\nonumber   &&+r^{s-1}\int_{S_r}[2m +\frac{s}{2} ]|\frac{\partial v_m}{\partial r}|^2e^{-2\rho}dx\\
\nonumber   &&+r^{s-1}\int_{S_r}[-t+\bar{a}_1+\frac{1}{2}\bar{\delta}_1+ 2\bar{\delta}\frac{m}{r}+o(1)]\frac{\partial v_m}{\partial r}v_m e^{-2\rho}dx\\
 \nonumber  && +r^{s-1}\int_{S_r}[  -\frac{a_4\bar{\delta}_1}{4}-\frac{s}{8}a_4^2+\frac{s}{2}\lambda-\frac{\bar{a}_2}{2}
 + o(1)]v_m^2e^{-2\rho}dx\\
  \label{Gsecondnonzerom} &&+r^{s-1}\int_{S_r}[\frac{s-2}{2} \frac{m(m+1)}{r^2}+\frac{\bar{\delta}^2}{2}\frac{m}{r^2}]v_m^2e^{-2\rho}dx
\end{eqnarray}
\begin{theorem}\label{Thmpositivesecond}
Under the conditions  of Theorem \ref{Mainthm2},
 there exist $s_0,R_0,m_0>0$ such that for $m\geq m_0$ and $r\geq R_0$,
 \begin{equation*}
    \frac{\partial F(m,r,t,s_0)}{\partial r} >0.
 \end{equation*}
\end{theorem}
\begin{proof}
Let $2<s_0<2a_3$, which will be specified later.  One has
\begin{equation*}
  | \frac{1}{2} v_m <\bar{\delta}_2 ,\nabla _{\omega}v_m>|\leq \frac{v_m^2}{\sigma}+(a_3-\frac{s_0}{2}+\epsilon(\sigma))|\nabla_{\omega}v_m|^2,
\end{equation*}
where $\sigma$ is any constant such that
\begin{equation*}
    \sigma<\frac{16a_3-8s_0}{|{\delta}_2|^2},
\end{equation*}
and $\epsilon(\sigma)>0$ is small.
 Let  $\sigma$ be close to $\frac{16a_3-8s_0}{|{\delta}_2|^2}$.
 By assumption \eqref{Gsecondlam1}, there exists $s_0$ with $2<s_0<2a_3$ such that
 \begin{equation*}
   -\frac{a_4\bar{\delta}_1}{4}-\frac{s_0}{8}a_4^2+\frac{s_0}{2}\lambda-\frac{\bar{a}_2}{2}-\frac{1}{\sigma}>0.
 \end{equation*}
 By Cauchy Schwartz inequality,
 we have
 \begin{equation}\label{equ1}
    4[\frac{s_0-2}{2} \frac{m(m+1)}{r^2}+\frac{\bar{\delta}^2}{2}\frac{m}{r^2} -\frac{a_4\bar{\delta}_1}{4}-\frac{s_0}{8}a_4^2+\frac{s_0}{2}\lambda-\frac{\bar{a}_2}{2}-\frac{1}{\sigma}][2m +\frac{s_0}{2}]>[-t+\bar{a}_1+\frac{1}{2}\bar{\delta}_1+2\bar{\delta}\frac{m}{r}]^2
 \end{equation}
 for large $m$ and large $r$.
 Putting all the estimates together, we obtain
 \begin{equation*}
   \frac{\partial F(m,r,t,s_0)}{\partial r}>0
 \end{equation*}
 for large $m$ and $r$.

\end{proof}
\begin{theorem}\label{Thmpositivezerosecond}
Under the assumptions of Theorem \ref{Mainthm2}, we have
\begin{equation*}
\frac{\partial F(0,r,0,s)}{\partial r}>0.
\end{equation*}
for large $r$.
\end{theorem}
\begin{proof}
Let $m,t=0$ in \eqref{Gsecondnonzerom}, one has
\begin{eqnarray*}
  \frac{\partial F(0,r,0,s)}{\partial r}
   &=& r^{s-1}\int_{S_r}[ (r(\nabla dr) -\frac{s}{2 }\hat{g})(\nabla v,\nabla v)+\frac{1}{2}r^{s-1} v <\bar{\delta}_2 ,\nabla _{\omega}v>]e^{-2\rho}dx\\
   &&+r^{s-1}\int_{S_r}[\frac{s}{2}]|\frac{\partial v}{\partial r}|^2e^{-2\rho}dx\\
   &&+r^{s-1}\int_{S_r}[ \bar{a}_1+\frac{1}{2}\bar{\delta}_1+o(1)]\frac{\partial v}{\partial r}v e^{-2\rho}dx\\
   && +r^{s-1}\int_{S_r}[  -\frac{a_4\bar{\delta}_1}{4}-\frac{s}{8}a_4^2+\frac{s}{2}\lambda-\frac{\bar{a}_2}{2}+o(1)]v^2e^{-2\rho}dx
\end{eqnarray*}
By the proof of \eqref{equ1},  it suffices to show that
 \begin{equation*}
    4[-\frac{a_4\bar{\delta}_1}{4}-\frac{s}{8}a_4^2+\frac{s}{2}\lambda-\frac{\bar{a}_2}{2}-\frac{1}{\sigma}][\frac{s}{2}]>[\bar{a}_1+\frac{1}{2}\bar{\delta}_1]^2.
 \end{equation*}
 This holds by assumption \eqref{Gsecondlam2}.
\end{proof}

{\textbf{Proof of Theorem \ref{Mainthm2}}}
\begin{proof}
The proof follows from the proof of Theorem \ref{Mainthm1}.
We only need to replace  Theorems \ref{Thmpositive} and  \ref{Thmpositivezero}, with Theorems \ref{Thmpositivesecond} and \ref{Thmpositivezerosecond}.
\end{proof}
\section{Proof of Theorems   \ref{Mainthm3}, \ref{goodbound} and Corollary \ref{Cor4}}
Let $q_1=0$, $q_2=-\frac{a_4^2}{4}-\frac{a_4\bar{\delta}}{2r}-\frac{a_4a_5}{2r}-V_2$ and the others as before.  By \eqref{Gpartial}, we have
\begin{eqnarray}
  \label{Gpartialcase1third} \frac{\partial F(m,r,t,s)}{\partial r} 
   &=& r^{s-1}\int_{S_r}[ r(\nabla dr) -(\frac{s}{2 }+\frac{1}{2}\bar{\delta})\hat{g})(\nabla v_m,\nabla v_m)]e^{-2\rho}dx\\
 \label{Gpartialcase2third}  &&+r^{s-1}\int_{S_r}[2m -\frac{\bar{\delta}}{2}+\frac{s}{2}]|\frac{\partial v_m}{\partial r}|^2e^{-2\rho}dx\\
  \label{Gpartialcase3third} &&+r^{s-1}\int_{S_r}[\bar{a}_1-t+\bar{\delta}\frac{m} {r}+o(1)]\frac{\partial v_m}{\partial r}v_me^{-2\rho}dx\\
  \label{Gpartialcase4third} && +r^{s-1}\int_{S_r}[(\lambda-\frac{a_4^2}{4})(\frac{s}{2}+\frac{\bar{\delta}}{2})-\frac{\bar{a}_2}{2}-\frac{1}{4}a_4\bar{\delta}_1+o(1)]v_m^2e^{-2\rho}dx\\
   \label{Gpartialcase5third}&&+r^{s-1}\int_{S_r} \frac{m(m+1)}{r^2}[ \frac{s-2}{2}+\frac{\bar{\delta}}{2}]v_m^2 e^{-2\rho}dx.
\end{eqnarray}

\begin{theorem}\label{Thmpositivethird}
Under the conditions of Theorem \ref{Mainthm3},
 there exist $s_0,R_0,m_0>0$ such that for $m\geq m_0$ and $r\geq R_0$,
 \begin{equation*}
    \frac{\partial F(m,r,t,s_0)}{\partial r} >0.
 \end{equation*}
\end{theorem}
\begin{proof}
Let  $s_0$ be such that  $s_0<2a_3-\delta$ and sufficiently close to $2a_3-\delta $.
By the assumption $a_3>1+\delta$ (see \eqref{Gconsthird}), one has
\begin{equation*}
     r(\nabla dr) -(\frac{s_0}{2 }+\frac{1}{2}\bar{\delta})\hat{g} \geq 0,
\end{equation*}
for large $r$, which implies
\begin{equation}\label{Gpartial2positivethird}
    \eqref{Gpartialcase1third}>0.
\end{equation}
By assumption $a_3>1+\delta$ (see \eqref{Gconsthird}) and $\lambda>\frac{a_4^2}{4}+\frac{2a_2+a_4\delta_1}{4(a_3-\delta)}$ (see \eqref{Gconlthird}), one has
\begin{equation*}
    (\lambda-\frac{a_4^2}{4})(\frac{s_0}{2}+\frac{\bar{\delta}}{2})-\frac{\bar{a}_2}{2}-\frac{a_4}{4}\bar{\delta}_1>0,
\end{equation*}
and
\begin{equation*}
\frac{s_0-2}{2}+\frac{\bar{\delta}}{2}>0.
\end{equation*}

By Cauchy-Schwartz inequality again, one has
\begin{equation*}
     [\frac{m(m+1)}{r^2} (\frac{s_0-2}{2}+\frac{\bar{\delta}}{2})+(\lambda-\frac{a_4^2}{4})(\frac{s_0}{2}+\frac{\bar{\delta}}{2})-\frac{\bar{a}_2}{2}-\frac{a_4}{4}\bar{\delta}_1]v_m^2
     +[2m -\frac{\bar{\delta}}{2}+\frac{s_0}{2}]|\frac{\partial v_m}{\partial r}|^2>|[\frac{\bar{a}_1}{2}-t+\bar{\delta}\frac{m} {r}]\frac{\partial v_m}{\partial r}v_m|
\end{equation*}
 for large $m$ and $r$. Thus, one has
 \begin{equation}\label{Gpartial5positivethird}
    |\eqref{Gpartialcase3third}|<\eqref{Gpartialcase2third}+\eqref{Gpartialcase4third}+\eqref{Gpartialcase5third}
 \end{equation}
for large $m$ and $r$.

Now   Theorem \ref{Thmpositivethird} follows from \eqref{Gpartial2positivethird} and \eqref{Gpartial5positivethird}.
\end{proof}
\begin{theorem}\label{Thmpositivezerothird}
Let  $s<\mu$ and $s$ be sufficiently close to $\mu$. Then
under the conditions of Theorem \ref{Mainthm3},
 we have
 \begin{equation*}
    \frac{\partial F(0,r,0,s)}{\partial r} >0,
 \end{equation*}
 for large $r$.
\end{theorem}
\begin{proof}
Let $m=0, t=0$ in \eqref{Gpartialcase1third}, one has
\begin{eqnarray*}
   \frac{\partial F(0,r,0,s)}{\partial r}
   &=&\int_{S_r} r^{s-1}[ r(\nabla dr) -(\frac{s}{2 }+\frac{1}{2}\bar{\delta})\hat{g})(\nabla v,\nabla v)]e^{-2\rho}dx\\
   &&+r^{s-1}\int_{S_r}[ -\frac{\bar{\delta}}{2}+\frac{s}{2}]|\frac{\partial v}{\partial r}|^2e^{-2\rho}dx\\
   &&+r^{s-1}\int_{S_r}[\bar{a}_1+o(1)]\frac{\partial v}{\partial r}v e^{-2\rho}dx\\
   && +r^{s-1}\int_{S_r}[(\lambda-\frac{a_4^2}{4})(\frac{s}{2}+\frac{\bar{\delta}}{2})-\frac{\bar{a}_2}{2}-\frac{a_4}{4}\bar{\delta}_1+o(1)]v^2
\end{eqnarray*}
We will show  that for large $r$,
\begin{equation*}
    \frac{\partial F(0,r,0,s)}{\partial r}>0.
\end{equation*}
By Cauchy Schwartz inequality, it suffices to prove
\begin{equation}\label{Gsolveeqthird}
    4[(\lambda-\frac{a_4^2}{4})(\frac{s}{2}+\frac{\bar{\delta}}{2})-\frac{\bar{a}_2}{2}-\frac{a_4}{4}\bar{\delta}_1][-\frac{\bar{\delta}}{2}+\frac{s}{2}]>
    |\bar{a}_1|^2,
\end{equation}
which holds by  assumption \eqref{Gconl} and $s$ is close to $\mu$.
\end{proof}
{\textbf{Proof of Theorem \ref{Mainthm2}}}
\begin{proof}
The proof follows from the proof of Theorems \ref{Mainthm1} and \ref{Mainthm2}.
We only need to replace  Theorems \ref{Thmpositive} and  \ref{Thmpositivezero}, with Theorems \ref{Thmpositivethird} and \ref{Thmpositivezerothird}.
\end{proof}

{\textbf{Proof of Corollary \ref{Cor4}}}
\begin{proof}
The proof follows from Theorem \ref{Mainthm2} and Lemma \ref{Lekey3}.
\end{proof}

{\textbf{Proof of Theorem \ref{goodbound}}}

In the proof Theorems \ref{Mainthm1} and  \ref{Mainthm3}, we let  $q_2=-\frac{a_4^2}{4}-\frac{a_4a_5}{2r}-V_2$    and
$q_2=-\frac{a_4^2}{4}-\frac{a_4a_5}{2r}-V_2-\frac{a_4\bar{\delta}}{2r}$ respectively.
Now we only need to let $q_2=-\frac{a_4^2}{4}-\frac{a_4a_5}{2r}-V_2-(1-\sigma)\frac{a_4\bar{\delta}}{2r}$, and following the proof of
Theorems \ref{Mainthm1} and \ref{Mainthm3}, we can prove Theorem \ref{goodbound}. We omit the details.

 \section*{Acknowledgments}

I would like to thank Svetlana Jitomirskaya for introducing to me paper \cite{kumura2010radial} and
inspiring discussions on this subject. I also want to thank Songying Li for telling me details about the unique continuation theorem.
  The author  was supported by the AMS-Simons Travel Grant 2016-2018 and  NSF DMS-1700314. This research was also
partially supported by NSF DMS-1401204.


\end{document}